\documentclass[12pt,twoside,a4paper]{amsart}

\usepackage[margin=2.5cm]{geometry}
\usepackage{amssymb,mathtools}
\usepackage{amsrefs}
\usepackage{mathrsfs}
\usepackage[dvipsnames]{xcolor}
\usepackage{hyperref}
\usepackage{float}
\usepackage{graphicx}
\usepackage{tikz}
\usetikzlibrary{patterns}
\usepackage{subcaption}
\usepackage{relsize}

\newtheorem{theorem}{Theorem}[section]
\newtheorem{lemma}[theorem]{Lemma}

\newtheorem{corollary}[theorem]{Corollary}
\theoremstyle{definition}

\theoremstyle{remark}
\newtheorem{remark}[theorem]{Remark}

\numberwithin{equation}{section}
\allowdisplaybreaks

\newcommand{\R}{\mathbb{R}}

\newcommand{\Z}{\mathbb{Z}}

\def\supp{{\rm supp\,}}
\usepackage{esint}

\title[Power-Weight Estimates for the Fractional Spherical Maximal Operator]{Sharp Power-Weight Estimates for the Fractional Spherical Maximal Operator}

\author[R. Basak]{Riju Basak}
\address[R. Basak]{Department of Mathematics, National Taiwan Normal University, No. 88, Section 4, Tingzhou Road, Wenshan District, Taipei City, Taiwan 116, R.O.C.
}
\email{rijubasak52@gmail.com}

\author[S.S. Choudhary]{Surjeet Singh Choudhary}
\address[S.S. Choudhary]{Mathematics Division, National Center for Theoretical Sciences, NTU, Cosmology Building, No. 1, Sec. 4, Roosevelt Rd., Taipei City, Taiwan 106, R.O.C.
}
\email{surjeet@ncts.ntu.edu.tw}

\author[C.Y. Shen]{Chun-Yen Shen}
\address[C.Y. Shen]{Department of Mathematics, National Taiwan University, No. 1, Sec. 4, Roosevelt Rd., Taipei City, Taiwan 106, R.O.C.
}
\email{cyshen@ntu.edu.tw}

\author[D. Spector]{Daniel Spector}
\address[D. Spector]{Department of Mathematics, National Taiwan Normal University, No. 88, Section 4, Tingzhou Road, Wenshan District, Taipei City, Taiwan 116, R.O.C.
\newline
Department of Mathematics, University of Pittsburgh, Pittsburgh, PA 15261, USA
}
\email{spectda@gapps.ntnu.edu.tw}


\begin{document}

\begin{abstract}
For $0\leq\alpha<n$, let $A_{*}^{\alpha}$ denote the fractional spherical maximal operator.  
In this paper we characterise the range of $\beta$ for which the $L^p-L^q$ power-weight estimate 
	\[\left\|A_{*}^{\alpha} f\right\|_{L^q(|x|^{q\beta})} \leq C \left\|f\right\|_{L^p(|x|^{p\beta})}\]  
holds.
In the diagonal case $p=q$, our work gives an affirmative answer to a conjecture of Lee.  
We prove analogous results for the lacunary fractional spherical maximal operator.  
\end{abstract}
	
	\maketitle
	
	\section{Introduction and Main Results}
	
	For a function $f \in \mathcal{S}(\mathbb{R}^n)$ and $0\leq \alpha< n$, we define the fractional spherical average as
	\[A_t^\alpha f(x):=t^{\alpha}\int_{S^{n-1}} f(x-ty) \, d\sigma(y)=t^{\alpha}f\ast\sigma_t(x),\]
	where $d\sigma$ is the normalised surface measure on $S^{n-1}$ and $\sigma_t$ is the normalised surface measure on the sphere of radius $t$.  Consider the fractional spherical maximal function, its lacunary counterpart, and the local spherical maximal function, defined by
	\begin{align}
	A_{*}^{\alpha}f(x)=\sup_{t>0} \big|A_t^\alpha f(x)\big|, \quad A_{lac}^{\alpha}f(x)=\sup_{k\in \mathbb{Z}} \big|A_{2^k}^\alpha f(x)\big|, \quad \widetilde{A}_{*}f(x)=\sup_{t \in [1,2]} \big|A_t f(x)\big|, \label{local_spherical}
	\end{align}
	respectively.  

Classical results of Stein \cite{Stein} and Bourgain \cite{Bourgain} assert that for $\alpha=0$, the spherical maximal operator $A_{*}\equiv A_{*}^{0}$ extends as a bounded operator on $L^p(\mathbb{R}^n)$ for $n/(n-1)<p\leq +\infty$.  Analogous results for the lacunary spherical maximal operator were established by Coifman and Weiss \cite[p.~246]{CoifmanWeiss} and C.P. Calder\'on \cite{Calderon_lacunary}.  For $\alpha \in (0,n)$, the study of $L^p(\R^n)\to L^q(\R^n)$ boundedness of $A_{*}^{\alpha}$ was initiated by Oberlin \cite{Oberlin}, while estimates in the interior of the sharp region of boundedness can be deduced from the work of Schlag \cite{Schlag}, Schlag and Sogge \cite{Schlag-Sogge}, see e.g. \cite[p.~534]{KS}; we refer the reader to the introduction of \cite{BCS} for a less abridged account of the history.

In the recent paper \cite{BCS}, the first, second, and fourth-named authors established restricted weak-type estimates for the fractional spherical maximal operator and its lacunary analogue on the boundary of their maximal type-sets.  Precisely, for $O=(0,0), P=(\frac{n-1}{n}, \frac{n-1}{n}), Q=(\frac{n-1}{n}, \frac{1}{n})$, $R=(\frac{n^2-n}{n^2+1}, \frac{n-1}{n^2+1})$, $A=(1,1), B=(\frac{n}{n+1}, \frac{1}{n+1})$, the authors extended the classical strong-type estimates in the interior to restricted weak-type estimates on the boundary of the following regions:
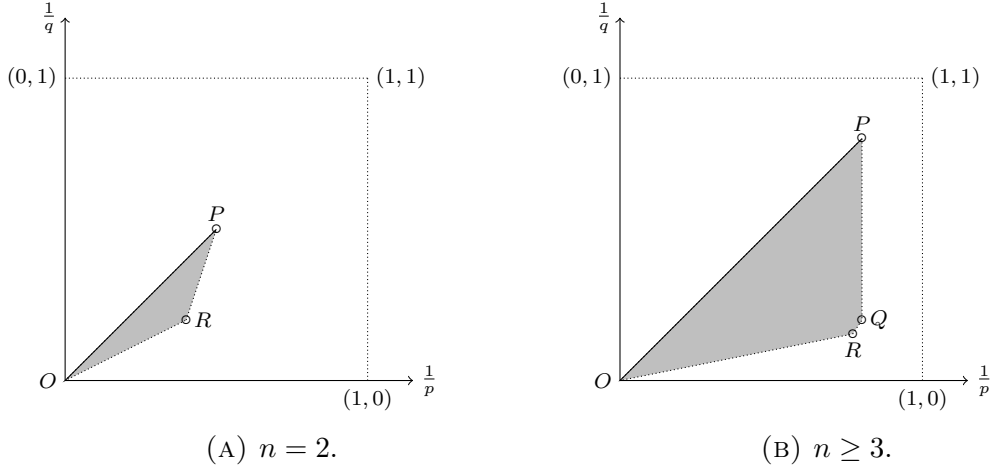
\begin{figure}[H]
		\centering
		\begin{subfigure}[b]{0.45\linewidth}
			\begin{tikzpicture}[scale=4]
				\tiny
				\fill[lightgray] (0,0)--(0.5,0.5)--(2/5,1/5)--cycle;
				\draw[thin][->]  (0,0)node[left]{$O$} --(1.15,0) node[right]{$\frac{1}{p}$};
				\draw[thin][->]  (0,0) --(0,1.2) node[left]{$\frac{1}{q}$};
				\draw[densely dotted] (0,1)node[left]{$(0,1)$}--(1,1)node[right]{$(1,1)$};
				\draw [densely dotted] (1,0)node[below]{$(1,0)$} --(1,1);
				\draw[densely dotted] (2/5,1/5)node[right]{$R$}--(0.5,0.5)node[above]{$P$}--(0,0)--cycle;
				\node at (0.5,0.5) {\tiny{$\circ$}};
				\node at (2/5,1/5) {\tiny{$\circ$}};
				\draw[thin] (0,0) --(0.5,0.5);
			\end{tikzpicture}
			\caption{$n= 2$.}	\label{Fig:full2}
		\end{subfigure}	
		\begin{subfigure}[b]{0.45\linewidth}
			\begin{tikzpicture}[scale=4]
				\tiny
				\fill[lightgray] (0,0)--(0.8,0.8)--(0.8,0.2)--(20/26,4/26)--cycle;
				\draw[thin][->]  (0,0)node[left]{$O$} --(1.15,0) node[right]{$\frac{1}{p}$};
				\draw[thin][->]  (0,0) --(0,1.2) node[left]{$\frac{1}{q}$};
				\draw[densely dotted] (0,1)node[left]{$(0,1)$}--(1,1)node[right]{$(1,1)$};
				\draw [densely dotted] (1,0)node[below]{$(1,0)$} --(1,1);
				\draw[densely dotted] (20/26,4/26)node[below]{$R$}--(0.8,0.2)node[right]{$Q$}--(0.8,0.8)node[above]{$P$}--(0,0)--cycle;
				\node at (0.8,0.8) {\tiny{$\circ$}};
				\node at (4/5,1/5) {\tiny{$\circ$}};
				\node at (20/26,4/26) {\tiny{$\circ$}};
				\draw[thin] (0,0) --(0.8,0.8);
			\end{tikzpicture}
			\caption{$n\geq 3$.}
		\end{subfigure}	
		\caption{Regions of $L^p(\R^n)\to L^q(\R^n)$ boundedness of $A_{*}^{\alpha}$.}\label{Fig:full}
	\end{figure}
	
	\begin{figure}[H]
		\centering
		\begin{tikzpicture}[scale=5]
			\tiny
			\fill[lightgray] (0,0)--(0.8,0.2)--(1,1)--cycle;
			\draw[thin][->]  (0,0)node[left]{$O$} --(1.15,0) node[right]{$\frac{1}{p}$};
			\draw[thin][->]  (0,0) --(0,1.2) node[left]{$\frac{1}{q}$};
			\draw[densely dotted] (0,1)node[left]{$(0,1)$}--(1,1)node[right]{$(1,1)$};
			\draw [densely dotted] (1,0)node[below]{$(1,0)$} --(1,1);
			\draw[densely dotted] (0.8,0.2)node[right]{$B$}--(1,1)node[above]{$A$}--(0,0)--cycle;
			\node at (1,1) {\tiny{$\circ$}};
			\node at (0.8,0.2) {\tiny{$\circ$}};
			\draw[thin] (0,0) --(1,1);
		\end{tikzpicture}
		\caption{Region of $L^p(\R^n)\to L^q(\R^n)$ boundedness of $A_{lac}^{\alpha}$.}
		\label{Fig:lac}
	\end{figure}
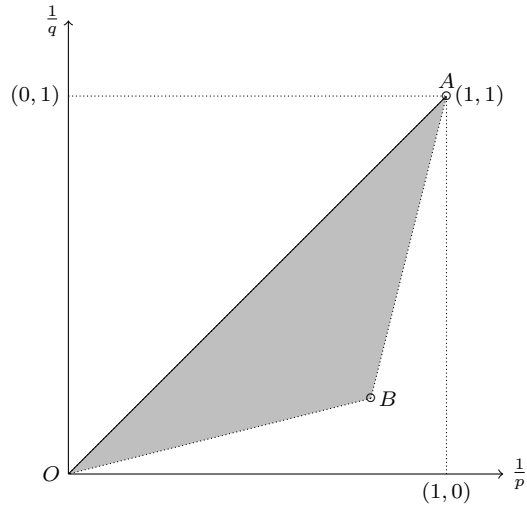
	
In this paper, we are interested in the question of power-weighted estimates for these operators, i.e. of the validity of the inequalities
	\begin{equation}\label{eq:main}
		\left\|A_{*}^{\alpha} f\right\|_{L^q(|x|^{q\beta})} \leq C \left\|f\right\|_{L^p(|x|^{p\beta})}
	\end{equation}
and
	\begin{equation}\label{eq:main_lac}
		\left\|A_{lac}^{\alpha} f\right\|_{L^q(|x|^{q\beta})} \leq C \left\|f\right\|_{L^p(|x|^{p\beta})}
	\end{equation}
for an admissible range of $\beta$.  The history of such estimates in the literature begins with the work of Rubio de Francia \cite{Rubio} and Duoandikoetxea and Vega \cite{DV}, who proved $L^p(|x|^{p\beta})\to L^p(|x|^{p\beta})$ estimates for the spherical maximal function for $n/(n-1)<p\leq +\infty$ when
\[-\frac{n-1}{p}<\beta<n-1-\frac{n}{p},\]
and for its lacunary counterpart for $1<p\leq +\infty$ when
\[-\frac{n-1}{p}\leq \beta<\frac{n-1}{p'}.\]
The results in \cite{Rubio} concern positive exponents for the full spherical maximal function, while \cite{DV} gives the extended range of negative exponents, the results for the lacunary spherical maximal function, and examples which show the ranges of $\beta$ are also necessary, except possibly the lower endpoint $\beta= (1-n)/p$ for the full spherical maximal function.  When $p\geq 2$, the validity of the estimate for this lower endpoint was subsequently proved by Lee \cite{Lee}, who conjectured that this restriction on $p$ is not necessary.  For further results concerning diagonal estimates, we refer the reader to \cite{DS, FRS} for power-weighted estimates for maximal functions defined by general dilation sets, to \cite{DMO, NRS} for power-weighted estimates for radial functions, and to \cite{Lacey} for $A_p$-weighted estimates by sparse domination.

Concerning off-diagonal estimates, Cowling, Garc\'ia-Cuerva, and Gunawan \cite{CGG} proved the following theorem in the more general setting of $A_p$ weights. 

	\begin{theorem}[\cite{CGG}, Theorem 4.3]
		\label{thm:CGG}
		Let $n\geq 3$. Assume that $\frac{n}{n-1}<p<q<n$ and $\alpha=\frac{n}{p}-\frac{n}{q}.$ For
		\[\max\left\{0,1-\frac{q}{p'}\right\}<\gamma\leq 1-\frac{q}{n},\]
		define
		\[s=\frac{q+2p'\gamma-p'}{p'\gamma}.\]
		If $w\in A_s$, then
		\[\left\|A_{*}^{\alpha} f\right\|_{L^q(w^\gamma)} \leq C \left\|f\right\|_{L^p(w^{p\gamma/q})}.\]
	\end{theorem}
An application of their result to power-weights implies the estimate \eqref{eq:main} for a restricted range of parameters:
	\begin{corollary}\label{CGG_theorem_cor}
		Let $n\geq 3$, and suppose $\frac{n}{n-1}<p<q<n$ and $\alpha=\frac{n}{p}-\frac{n}{q}.$ Then, for every
		\[-\frac{n}{q}+1<\beta<n-1-\frac{n}{p},\]
		\eqref{eq:main} holds for $A_{*}^{\alpha}$.
	\end{corollary}	
	
The state of the art from the papers \cite{CGG,DV,Lee,Rubio} is depicted in the following figure.

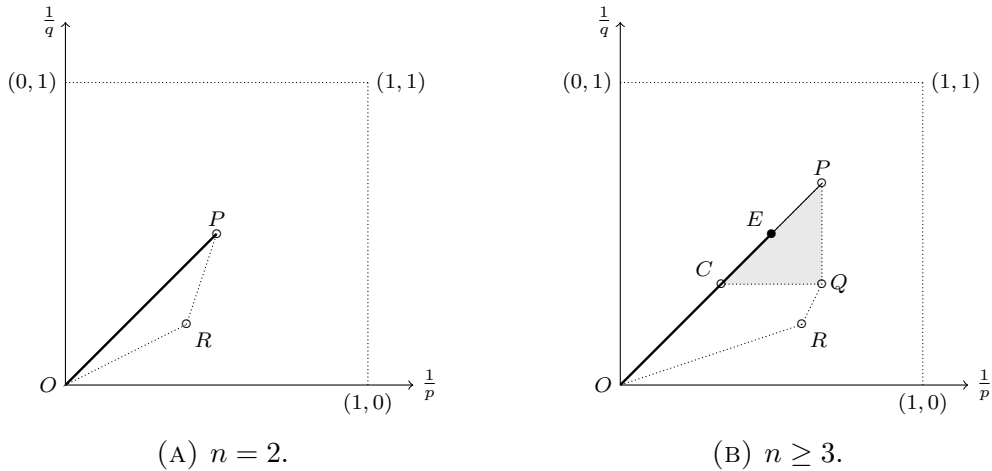
\begin{figure}[H]
	\centering
	\begin{subfigure}[b]{0.45\linewidth}
		\centering
		\begin{tikzpicture}[scale=4]
			\tiny
			\draw[thin,->] (0,0)node[left]{$O$} --(1.15,0) node[right]{$\frac{1}{p}$};
			\draw[thin,->] (0,0) --(0,1.2) node[left]{$\frac{1}{q}$};
			\draw[densely dotted] (0,1)node[left]{$(0,1)$}--(1,1)node[right]{$(1,1)$};
			\draw[densely dotted] (1,0)node[below]{$(1,0)$} --(1,1);
			\draw[densely dotted] (2/5,1/5)node[below right]{$R$}--(1/2,1/2)node[above]{$P$}--(0,0)--cycle;
			\draw[line width=0.9pt] (0,0) -- (1/2,1/2);
			\node at (1/2,1/2) {\tiny{$\circ$}};
			\node at (2/5,1/5) {\tiny{$\circ$}};
		\end{tikzpicture}
		\caption{$n=2$.}\label{Fig:full2_weight_before}
	\end{subfigure}
	\begin{subfigure}[b]{0.45\linewidth}
		\centering
		\begin{tikzpicture}[scale=4]
			\tiny
			\fill[lightgray!35] (1/3,1/3)--(2/3,1/3)--(2/3,2/3)--cycle;
			\draw[thin,->] (0,0)node[left]{$O$} --(1.15,0) node[right]{$\frac{1}{p}$};
			\draw[thin,->] (0,0) --(0,1.2) node[left]{$\frac{1}{q}$};
			\draw[densely dotted] (0,1)node[left]{$(0,1)$}--(1,1)node[right]{$(1,1)$};
			\draw[densely dotted] (1,0)node[below]{$(1,0)$} --(1,1);
			\draw[densely dotted] (3/5,1/5)node[below right]{$R$}--(2/3,1/3)node[right]{$Q$}--(2/3,2/3)node[above]{$P$}--(0,0)--cycle;
			\draw[densely dotted] (1/3,1/3)--(2/3,1/3);
			\draw[thin] (0,0)--(2/3,2/3);
			\draw[line width=0.9pt] (0,0) -- (1/2,1/2);
			\node at (2/3,2/3) {\tiny{$\circ$}};
			\node at (2/3,1/3) {\tiny{$\circ$}};
			\node at (3/5,1/5) {\tiny{$\circ$}};
			\node at (1/3,1/3) {\tiny{$\circ$}};
			\node[above left] at (1/3,1/3) {$C$};
			\node at (1/2,1/2) {\tiny{$\bullet$}};
			\node[above left] at (1/2,1/2) {$E$};
		\end{tikzpicture}
		\caption{$n\geq3$.}
	\end{subfigure}
	\caption{Encodes known $L^p(|x|^{p\beta})\to L^q(|x|^{q\beta})$ boundedness of $A_{*}^{\alpha}$.}
	\label{Fig:full_weight_before}
\end{figure}

Figure \ref{Fig:full_weight_before} requires some explanation, as it projects the three dimensional region depending on $(1/p,1/q,\beta)$ onto a two dimensional picture.  For $n=2$, when $p=q$ the complete range of $\beta$ is characterised.  When $p<q$, only the unweighted estimate $\beta=0$ was known.  For $n\geq3$, when $p=q$ the complete range of $\beta$ is characterised for $p \geq 2$, the segment $OE$, with up to the endpoint results for $ n/(n-1)<p<2$, the segment $PE$.  When $p<q$, the triangle $CPQ$ corresponds to the result of Cowling, Garc\'ia-Cuerva, and Gunawan, where from their result one has a partial range of $\beta$.

In particular, for $n=2$ there are no existing off-diagonal estimates aside from the known unweighted estimate, while for $n \geq 3$ the lower endpoint value of $\beta$ on $PE$, a range of exponents in the negative regime for the off-diagonal estimates in the region $CPQ$, and all weighted estimates in the quadrilateral $OCQR$ remain open.  Our main result is the following theorem.		\begin{theorem}\label{thm:full}
		Let $n\geq 2$ and $\frac{n}{n-1}< p\leq q<\infty$ with $\frac{\alpha}{n}=\frac{1}{p}-\frac{1}{q}$. Then
		\[\|A_{*}^{\alpha} f\|_{L^q(|x|^{q\beta})} \leq C \|f\|_{L^p(|x|^{p\beta})},\]
        holds for
		\begin{enumerate}
			\item $\big(\frac{1}{p},\frac{1}{q}\big)\in OPQ\cup OP\cup OQ$, if and only if $-\frac{n}{q}+\frac{1}{p}\leq \beta <n-1-\frac{n}{p}$,
			\item $\big(\frac{1}{p},\frac{1}{q}\big)\in OQR$, if and only if $-\frac{n}{q}+\frac{1}{p}\leq \beta\leq n-1+\frac{n-1}{q}-\frac{n+1}{p}$.
		\end{enumerate}
	\end{theorem}
Theorem \ref{thm:full} gives a complete characterisation of the validity of the estimate for the fractional spherical maximal operator in the region for the full range of $\beta$.  In particular, it implies the known consequence of Cowling, Garc\'ia-Cuerva, and Gunawan \cite{CGG},  the diagonal estimates of Duoandikoetxea and Vega \cite{DV}, and the endpoint result of Lee \cite{Lee}.  Moreover, it gives an affirmative answer to the conjecture of Lee following his Theorem 1.1, see also  Problem 6.1 of Roos and Seeger \cite{RoosSeeger}.  

In the lacunary setting we have the following theorem.
		\begin{theorem}\label{thm:lac}
		Let $n\geq 2$ and $1< p\leq q<\infty$ with $\frac{\alpha}{n}=\frac{1}{p}-\frac{1}{q}$. If $\big(\frac{1}{p},\frac{1}{q}\big)$ lies in the interior of the region $OAB$ or the open line segment $OA$, then
		\[\|A_{lac}^{\alpha} f\|_{L^q(|x|^{q\beta})} \leq C \|f\|_{L^p(|x|^{p\beta})},\]
		if and only if $-\frac{n}{q}+\frac{1}{p}\leq \beta < n - 1 + \frac{1}{q} - \frac{n}{p}$.
	\end{theorem}
\noindent
Theorem \ref{thm:lac} completely characterises the validity of power-weighted estimates for the lacunary fractional spherical maximal operator.
	
The plan of the paper is as follows. In Section \ref{Prelim}, we discuss some preliminary results. In Sections \ref{single average}, \ref{lacunary} and \ref{full}, we prove power-weighted estimates for the single average, the lacunary fractional spherical maximal function (Theorem \ref{thm:lac}), and the full fractional spherical maximal function (Theorem \ref{thm:full}), respectively. In Section \ref{examples}, we give examples to obtain necessary conditions.

		\section{Preliminaries}\label{Prelim}
		In this section we collect several results that will be useful in the sequel.  First, we note that regions referred to without an overline are open regions, while those with an overline are closed.  For example,  in the lacunary case, $OAB$ refers to the open triangle, while $\overline{OAB}$ to the closed triangle.   The triangle $OAB$ does not contain the edges $OA, AB, OB$, which in turn, do not include their boundary points $O, A, B$.  This is in contrast to $\overline{OA}, \overline{AB}, \overline{OB}$, which contain the points $O, A, B$.  Second, it is useful to note that the localization of the fractional spherical maximal operator is comparable to the local spherical maximal operator, i.e.
	\begin{equation}\label{localpointwise}
		\sup_{t\in[1,2]} \big|A_t^\alpha f(x)\big|\sim \widetilde{A}_{*}f(x).
	\end{equation}

Third, we discuss some easy consequences of $L^p$-boundedness of a sublinear operator $T$ in terms of power-weight estimates. We decompose $\mathbb R^n$ into dyadic annuli
	\[A_l=\{x\in\mathbb R^n:2^l\le |x|<2^{l+1}\},\]
	and write
	\[f=\sum_{l\in\mathbb Z}f_l,\qquad f_l=f\chi_{A_l}.\]
Then we have the following lemma which will be useful in settling several easy subcases that arise in the estimates.
	\begin{lemma}\label{direct}
		Let $C\in\Z$ be fixed. Suppose that a sublinear operator $T$ is bounded from $L^p(\R^n)$ to $L^q(\R^n)$, then
		\begin{align*}
			\left\|\sum_{k\in\mathbb Z}\chi_{A_k}T\Big(\sum_{l\ge k+C}f_l\Big)\right\|_{L^q(|x|^{q\beta})}&\lesssim \|f\|_{L^p(|x|^{p\beta})}&\text{when } \beta>0,\\
			\left\|\sum_{k\in\mathbb Z}\chi_{A_k}T\Big(\sum_{l\le k+C}f_l\Big)\right\|_{L^q(|x|^{q\beta})}&\lesssim \|f\|_{L^p(|x|^{p\beta})}&\text{when } \beta<0.
		\end{align*}
	\end{lemma}
	\begin{proof}
		We will prove the first estimate, the second follows in the same way. Since the annuli $A_k$ are disjoint and $|x|^\beta\sim 2^{k\beta}$ on $A_k$, we have
		\begin{align*}
			\left\|\sum_{k\in\mathbb Z}\chi_{A_k}T\Big(\sum_{l\ge k+C}f_l\Big)\right\|_{L^q(|x|^{q\beta})}&=\left(\sum_{k\in\Z}\left\|\chi_{A_k}T\Big(\sum_{l\ge k+C}f_{l}\Big)\right\|_{L^q(|x|^{q\beta})}^q\right)^{\frac 1q}\\
			&\lesssim \left(\sum_{k\in\Z}\left(2^{k\beta}\sum_{l\ge k+C}\Big\|Tf_{l}\Big\|_{L^q}\right)^q\right)^{\frac 1q}\\
			&\lesssim \left(\sum_{k\in\Z}\left(\sum_{l\ge k+C}2^{(k-l)\beta} \|f_{l}\|_{L^p(|x|^{p\beta})}\right)^q\right)^{\frac 1q}.
		\end{align*}    
		Define the sequences on $\Z$ by
		\begin{align*}
			F(m)=\|f_{m}\|_{L^p(|x|^{p\beta})},\qquad G(m)=2^{m\beta}\chi_{(-\infty,-C]}(m).
		\end{align*}
		When $\beta>0$, $G\in \ell_1(\Z)$ with a norm depending on $C$ and $\beta$. Also, we have $\|F\|_{\ell_p}=\|f\|_{L^p(|x|^{p\beta})}$. Hence Young's inequality on $\ell_q(\Z)$ and the $\ell_p\hookrightarrow \ell_q$ embedding give us
		\begin{align*}
			\left\|\sum_{k\in\mathbb Z}\chi_{A_k}T\Big(\sum_{l\ge k+C}f_l\Big)\right\|_{L^q(|x|^{q\beta})}&\lesssim \|F\ast G\|_{\ell_q}\\
			&\leq \|F\|_{\ell_q}\|G\|_{\ell_1}\\
			&\lesssim \|F\|_{\ell_p} = \|f\|_{L^p(|x|^{p\beta})}.
		\end{align*}
		For $\beta<0$, we use the sequence $G(m)=2^{m\beta}\chi_{[-C,\infty)}(m)$ instead and the result follows by an analogous argument.
	\end{proof}
	
	We will make repeated use of interpolation, including standard interpolations between Lebesgue spaces, complex interpolation and weighted real interpolation. We recall the following result from \cite[p. 17]{Freitag78}.
	\begin{lemma}\cite[Theorem 2]{Freitag78}\label{weightedinterpolation}
		Let $0<p_0<p_1<\infty$, $\frac{1}{p}=\frac{1-\theta}{p_0}+\frac{\theta}{p_1},\, 0<\theta<1$, and $0<q\le\infty$. For weights $w_0$ and $w_1$, we have
		\[\left(L^{p_0}(w_0,dx),L^{p_1}(w_1,dx)\right)_{\theta,q}=L^{p,q}\left(\left(\frac{w_1^{p_1}}{w_0^{p_0}}\right)^{\frac{1}{p_1-p_0}},\left(\frac{w_0}{w_1}\right)^{\frac{p_0p_1}{p_1-p_0}}dx\right).\]
\end{lemma}
\noindent
Here for a weight $u$ and a measure $v$, we use the notation $L^{p,q}(u,v)$ to denote the space associated to the quasi-norm defined by
\begin{align*}
\|f\|_{L^{p,q}(u,v)} = \|u f\|_{L^{p,q}(v)},
\end{align*}
where the right-hand-side is a Lorentz space defined in the standard way.

	Next, we state some results for the spherical average and the local spherical maximal function when the function is supported on a small ball.
	\begin{lemma}\label{singlelocalsmallsuport}
		Let $f$ be supported on $B(0, \varepsilon)$ with $\varepsilon \ll 1$ and $A_0=B(0,2)\setminus B(0,1)$. Then
		\[\Vert A_1 f\Vert_{L^q(A_0)} \lesssim \varepsilon^{n - 1 + \frac{1}{q} - \frac{n}{p}} \Vert f\Vert_{L^p},\]
		for $(1/p, 1/q)\in \overline{OAB}$.
	\end{lemma}
	
	\begin{proof}
		First, we point out a trivial estimate due to the support of the function $f$ and $x\in A_0$. We have 
		\[\Vert A_1 f\Vert_{L^\infty(A_0)} \lesssim  \varepsilon^{n-1} \Vert f\Vert_{L^\infty}.\]
		Also, for $(1/p, 1/q)\in \overline{AB}$, the boundedness of $A_1$ implies
		\[\Vert A_1 f\Vert_{L^q(A_0)} \lesssim  \Vert f\Vert_{L^p}.\]
		Also note that $n - 1 + \frac{1}{q} - \frac{n}{p}=0$ when $(1/p, 1/q)\in \overline{AB}$. Interpolating these estimates proves the Lemma for $(1/p, 1/q)\in \overline{OAB}$. 
	\end{proof}
	
	\begin{lemma}\label{localsmallsuport}
		Let $f$ be supported on $B(0, \varepsilon)$ with $\varepsilon \ll 1$ and $A_0=B(0,2)\setminus B(0,1)$. For $(1/p, 1/q)\in OPQR\cup OP$, define $\gamma = \min\left\{ n - 1 - \frac{n}{p}, \, n - 1 + \frac{n-1}{q} - \frac{n+1}{p} \right\}$. Then for every $\eta>0$
		\[\Vert\widetilde{A}_* f\Vert_{L^q(A_0)} \lesssim \varepsilon^{\gamma-\eta} \Vert f\Vert_{L^p}.\]
	\end{lemma}
	\begin{proof}
		Due to the support of the function $f$ and $x\in A_0$, we have 
		\[\Vert\widetilde{A}_* f\Vert_{L^\infty(A_0)} \lesssim  \varepsilon^{n-1} \Vert f\Vert_{L^\infty}.\]
		Also, for $(1/p, 1/q)\in PQ \bigcup QR\cup Q$, the restricted weak-type boundedness of $\widetilde{A}_*$ (see \cite[Theorems 1.1 and 1.4]{Lee1}) asserts
		\[\Vert\widetilde{A}_* f\Vert_{L^{q,\infty}(A_0)} \lesssim  \Vert f\Vert_{L^{p,1}}.\]
		Note that $\gamma=0$ when $(1/p, 1/q)\in PQ\bigcup QR$. Applying real interpolation between these estimates proves the Lemma for $(1/p, 1/q)\in OPQR$.  For $n\geq 3$, the fact that one has a restricted weak-type at $P$ (see \cite{Bourgain0}) yields the conclusion for the line segment $OP$ by the same argument.  When $n=2$, one instead interpolates with a strong-type estimate for a point on $OP$ arbitrarily close to $P$ to obtain the claimed result.
	\end{proof}
	\begin{remark}\label{rem:dim2}
	As one observes from the proof, the loss of $\eta>0$ only appears for $n=2$ on the segment $OP$.
\end{remark}
		
	We next consider the dual situation in which the local spherical maximal function is restricted on a small ball. Let $\psi \in \mathscr{S}(\mathbb{R}^n)$ such that $\supp(\widehat{\psi})\subset B(0,4)\setminus B(0,1/2)$ and
	\begin{equation}\label{identity}
		\sum_{j\in\Z} \widehat{\psi}_{2^j}(\xi)=1, \quad \xi\neq 0,
	\end{equation}
	where $\psi_t(x)=t^n \psi(tx)$.  Further define $\widetilde{\psi}_{2^{j}} = \psi_{2^{j-1}} +\psi_{2^{j}} +\psi_{2^{j+1}}$.
	
	Using the partition of unity from \eqref{identity}, we have the following decomposition of $\widetilde{A}_{*}$ :
	\begin{equation}\label{eq:deco}
		\widetilde{A}_{*}f(x)\leq \sup_{t\in[1,2]}\left|f\ast\sum_{j\leq0}\psi_{2^{j}}\ast\sigma_{t}(x)\right|+\sum_{j\geq1}\sup_{t\in[1,2]}\big|f\ast\psi_{2^{j}}\ast\sigma_{t}(x)\big|=:\sum_{j\geq0}M_{j}f(x),
	\end{equation}
	where
    \[M_0f(x)=\sup_{t\in[1,2]}\left|f\ast\sum_{j\leq0}\psi_{2^{j}}\ast\sigma_{t}(x)\right|,\]
    and for $j\geq1$,
	\[M_{j}f(x)=\sup_{t\in[1,2]}\big|f\ast\psi_{2^{j}}\ast\sigma_{t}(x)\big|.\]

	\begin{lemma}[\cite{BCS}, Lemma 4.1]\label{Pointwise:full}
		Let $n\geq2$ and $0\leq\alpha\leq n$. Then for $j\geq 1$, we have
		\[M_{j}f(x)\lesssim 2^j\mathcal{M}^{\alpha}f(x).\]
	\end{lemma}
	\noindent
	Here we use the notation
	\begin{align*}
	\mathcal{M}^{\alpha}f(x):= \sup_{r>0} r^\alpha \fint_{B(x,r)} |f(y)|\;dy
	\end{align*}
	to denote the fractional maximal function of $f$.

	The following lemma tells us that we may consider the averaging operator to control the maximal function.
	\begin{lemma}[\cite{Lee1}, Lemma 1.3]\label{sobo}
		Let $1< p\leq \infty$ and $F : \R^n\times [1,2]\to \R$ be a smooth function. Then,
		\[\left\|\sup_{t\in[1,2]}|F(x,t)|\right\|_{L^p}\lesssim \| F\|_{L^p(\R^n\times[1,2])}+\| F\|_{L^p(\R^n\times[1,2])}^{\frac{1}{p'}}\|\partial_t F\|_{L^p(\R^n\times[1,2])}^{\frac{1}{p}}.\]
	\end{lemma}
	
	Recall the Fourier multiplier representation of the spherical average
	\[A_t f(x) = \int_{\R^n} \widehat{f}(\xi) \widehat{d\sigma}(t\xi) e^{ix\cdot\xi} \; d\xi,\]
	where $\widehat{d\sigma}$ is the Fourier transform of the surface measure on the unit sphere in $\R^n$ given by
	\[\widehat{d\sigma}(\xi) = \frac{J_{\frac{n-2}{2}}(|\xi|)}{|\xi|^{\frac{n-2}{2}}}.\]
	Here, $J_{\frac{n-2}{2}}$ denotes the standard Bessel function of order $\frac{n-2}{2}$. Using the asymptotics of the Bessel function (see page 338 of  \cite{book-Stein}), we can write
	\begin{equation}\label{measure}
		\widehat{d\sigma}(\xi) = \sum_{\pm, 0\le m\le N} a_m^\pm e^{\pm i|\xi|}|\xi|^{-\frac{n-1}{2}-m}+ E_N(|\xi|),\qquad |\xi|\geq1,
	\end{equation}
	for some constants $a_m^\pm$, and $E_N$ is a smooth function satisfying
	\[\bigg|\Big(\frac{d}{dr}\Big)^\vartheta E_N(r)\bigg|\lesssim r^{-\big(\frac{n+1+2N}{2}+\gamma\big)}, \qquad 0\le \vartheta \le \frac{N+1}{2}.\]
	By choosing $N > 0$ sufficiently large, the error term $E_N$ gives a negligible contribution. Furthermore, in view of the expansion \eqref{measure}, every term with $m\geq1$ has additional decay. Thus it is enough to consider the leading oscillatory multiplier $e^{i|\xi|}|\xi|^{-\frac{n-1}{2}}$. The term with the negative sign can be handled identically.
	
	\begin{lemma}\label{localized estimate}
		Let $j\ge 1$ and $0< \varepsilon \le 1$. Then 
		\[\|\chi_{B(0,\varepsilon)}M_{j}f\|_{L^n} \lesssim \varepsilon^{1/n}\|f\|_{L^{\frac{n}{n-1}}}\]
		holds with a constant independent of $j$ and $\varepsilon$.
	\end{lemma}
	\begin{proof}
		Applying Lemma~\ref{sobo} with $p=2$ gives
		\begin{align}\label{localaverage}
			\|\chi_{B(0,\varepsilon)}M_{j}f\|_{L^2}\lesssim&\|\chi_{B(0,\varepsilon)}|f\ast\psi_{2^j}\ast\sigma_t|\|_{L^2(\R^n\times[1,2])}\nonumber\\
			&+\|\chi_{B(0,\varepsilon)}|f\ast\psi_{2^j}\ast\sigma_t|\|_{L^2(\R^n\times[1,2])}^{1/2}\|\chi_{B(0,\varepsilon)} |\partial_t(f\ast\psi_{2^j}\ast\sigma_t)|\|_{L^2(\R^n\times[1,2])}^{1/2}.
		\end{align}
		On the support of $\widehat{\psi}_{2^j}$, we have $|\xi|\sim 2^j$. Hence the multiplier of the operator $2^{-j}\partial_t(f\ast\psi_{2^j}\ast\sigma_t)$ behaves like the multiplier of $f\ast\psi_{2^j}\ast\sigma_t$. Thus, by \eqref{measure} and the discussion thereafter, it is enough to estimate
		\[2^{-j\frac{n-2}{2}}\|\chi_{B(0,\varepsilon)}\mathcal W_t f\|_{L^2(\R^n\times[1,2])},\]
		where $\mathcal W_t$ is the wave propagator given by
		\[\mathcal W_tg(x):=\int_{\R^n}e^{i(x\cdot\xi+t|\xi|)}\widehat g(\xi)\,d\xi.\]
		We claim that
		\begin{equation}\label{wave-local}
			\int_{\R}\int_{|x|\leq R}|\mathcal W_tf(x)|^2\,dx\,dt \lesssim R\|f\|_{L^2}^2,\qquad R>0.
		\end{equation}
		If we assume the claim, then  by \eqref{measure} and the discussion thereafter, we obtain
		\begin{equation}\label{L2-L2}
			\|\chi_{B(0,\varepsilon)}M_{j}f\|_{L^2}\lesssim 2^{-j\frac{n-2}{2}}\varepsilon^{1/2}\|f\|_{L^2}.
		\end{equation}
		On the other hand, Lemma~\ref{Pointwise:full} with $\alpha=n$ gives us 
		\begin{equation}\label{L1-Linfty}
			\|\chi_{B(0,\varepsilon)}M_{j}f\|_{L^\infty}\lesssim 2^{j}\|f\|_{L^1}.
		\end{equation}
		Interpolation between \eqref{L2-L2} and \eqref{L1-Linfty} yields
		\[\|\chi_{B(0,\varepsilon)}M_{j}f\|_{L^n} \lesssim \varepsilon^{1/n}\|f\|_{L^{\frac{n}{n-1}}}.\]
		This completes the proof of Lemma~\ref{localized estimate} except for the estimate \eqref{wave-local}.
		
		Writing the frequency integral of $\mathcal W_t$ in polar coordinates, we get
		\[\mathcal{W}_t f(x) = \int_0^\infty e^{itr} \int_{\mathbb S^{n-1}} e^{ix \cdot (r\omega)} \widehat{f}(r\omega) r^{n-1} \, d\sigma(\omega) dr.\]
		Let us define a function $H_x(r)$ for $r \ge 0$ by
		\[H_x(r) := \int_{\mathbb S^{n-1}} e^{irx \cdot \omega} \widehat{f}(r\omega) r^{n-1} \, d\sigma(\omega),\]
		and extend $H_x(r) = 0$ for $r < 0$. We can see $\mathcal{W}_t f(x)$ as the 1-dimensional inverse Fourier transform of $H_x(r)$ evaluated at time $t$ (up to a constant multiple of $2\pi$). 
		
		Applying Plancherel's theorem in the time variable $t$, we obtain
		\begin{align*}
			\int_{\R} |\mathcal{W}_t f(x)|^2 \, dt &= 2\pi \int_{\R} |H_x(r)|^2 \, dr \\
			&= 2\pi \int_0^\infty \bigg| \int_{\mathbb S^{n-1}} e^{irx \cdot \omega} \widehat{f}(r\omega) r^{n-1} \, d\sigma(\omega) \bigg|^2 dr
		\end{align*}
		Now, integrating over $|x| \le R$ and using Fubini's theorem, we get
		\begin{align*}
			\int_{|x| \le R} \int_{\R} |\mathcal{W}_t f(x)|^2 \, dt \, dx &= 2\pi \int_0^\infty \int_{|x| \le R} \bigg| \int_{\mathbb S^{n-1}} e^{irx \cdot \omega} \widehat{f}(r\omega) \, d\sigma(\omega) \bigg|^2 dx\, r^{2n-2} \, dr\\
			&= 2\pi \int_0^\infty \int_{|x| \le rR} \bigg| \int_{\mathbb S^{n-1}} e^{ix \cdot \omega} \widehat{f}(r\omega) \, d\sigma(\omega) \bigg|^2 dx\, r^{n-2}\,dr\\
			&= 2\pi \int_0^\infty \int_{|x| \le rR} |\widehat{f_r d\sigma}(x) |^2 dx\, r^{n-2}\,dr,
		\end{align*}
		where we applied a change of variables $rx\to x$ in the second equality and wrote $f_r(\omega) := \widehat{f}(r\omega)$.
		
		Invoking the classical Agmon-H\"ormander estimate for the sphere \cite[Theorem 2.1]{AH}, we get
		\begin{align*}
			\int_{|x| \le R} \int_{\R} |\mathcal{W}_t f(x)|^2 \, dt \, dx &\lesssim R \int_0^\infty \|f_r\|_{L^2(\mathbb S^{n-1})}^2 \, r^{n-1}\,dr\\
			&= R \int_0^\infty \int_{\mathbb S^{n-1}}|\widehat{f}(r\omega)|^2\,d\sigma(\omega)\, r^{n-1}\,dr\\
			&= R\|f\|^2_{L^2}.
		\end{align*}
		This proves the claim \eqref{wave-local}.
	\end{proof}
	We now prove a power-weighted estimate for the frequency localised operator $M_j$.
	\begin{lemma}\label{demposedlocalweighted}
		Let $k<0$ and $j\geq1$. Suppose $\operatorname{supp}(f)\subset A_0$ and fix $\beta_-=\frac{n}{q}-\frac{1}{p}$. For $\big(\frac{1}{p},\frac{1}{q}\big)\in OPQR \cup OP$, there exists $\epsilon>0$ such that
		\begin{equation*}
			\|\chi_{A_k}M_{j}f\|_{L^q(|x|^{-q\beta_-})}\lesssim 2^{-\epsilon|j+k|}\|f\ast\psi_{2^j}\|_{L^p}.
		\end{equation*}
	\end{lemma}
	\begin{proof}
		When $p=q>\frac{n}{n-1}$, the lemma was proved in \cite[Proposition 1.2]{Lee}. We have
		\begin{equation}\label{demposedweightedL2}
			\|\chi_{A_k}M_{j}f\|_{L^p(|x|^{-p\beta_-})}\lesssim 2^{-\delta|j+k|}\|f\ast\psi_{2^j}\|_{L^p},
		\end{equation}
		for some $\delta=\delta(n,p)>0$.
		
		If we show that for $\left(\frac{1}{p},\frac{1}{q}\right)\in PQ\cup QR\cup OR$, we have
		\begin{equation}\label{demposedweightedPQRO}
			\|\chi_{A_k}M_{j}f\|_{L^q(|x|^{-q\beta_-})}\lesssim \|f\ast\psi_{2^j}\|_{L^p}.
		\end{equation}
		Then, applying complex interpolation between \eqref{demposedweightedL2} and \eqref{demposedweightedPQRO} to the map
\[
g \mapsto \chi_{A_k} \sup_{t \in [1,2]} | g \ast \widetilde{\psi}_{2^j} \ast \sigma_t|
\]
acting on $g = f\ast \psi_{2^j}$ proves the Lemma.
		
		On the line segment $OR$, we have $\beta_-=0$. Thus,  unweighted estimates of $\tilde{A_*}$ from \cite[Theorem 1.4]{Lee1} gives us \eqref{demposedweightedPQRO}.
		
		For the line segment $PQ$, from Lemma~\ref{localized estimate} along with H\"older's inequality, we obtain
		\begin{align*}
			\|\chi_{A_k}M_{j}f\|_{L^q(|x|^{-q\beta_-})} &\lesssim 2^{-k\beta_-}|A_k|^{1/q-1/n} \|\chi_{B(0,2^{k+1})}M_{j}f\|_{L^n}\\
			&\lesssim 2^{-k\beta_-} 2^{k(n/q-1)} 2^{k/n}\|f\ast\psi_{2^j}\|_{L^{\frac{n}{n-1}}}\\
			&= 2^{k(-\beta_- + n/q - 1 + 1/n)}\|f\ast\psi_{2^j}\|_{L^{\frac{n}{n-1}}}\\
			&=\|f\ast\psi_{2^j}\|_{L^{\frac{n}{n-1}}}.
		\end{align*}
		Real interpolation of the estimates at $Q$ and $R$ proves \eqref{demposedweightedPQRO} for the line segment $QR$.
	\end{proof}
	
Finally, we end this section by proving a weighted boundedness estimate which will be useful to prove positive endpoint for range of $\beta$ in the boundedness region OQR.  Let $\left(\frac{1}{p},\frac{1}{q}\right)\in OQR$, and set $\alpha=\frac{n}{p}-\frac{n}{q}$ and $\beta_{+}= n-1+\frac{n-1}{q}-\frac{n+1}{p}$. 

For $m\geq 100$, let us define the positive sublinear operators
\begin{align}
\nonumber T_mg(x)=&|x|^{\beta_++\frac{n}{q}}\sum_{k\in\Z}\chi_{A_k}(x)
 A^{\alpha}_{*}\big(|\cdot|^{-\beta_+-\frac{n}{p}}|g|\chi_{A_{k-m}}\big)(x),\\
 \text{and} \qquad  Tg=&\sum_{m\geq 100}T_mg.\label{eq:T opeartor}
\end{align}

\begin{lemma}\label{lem:upper endpoint}
 Let $\left(\frac{1}{p},\frac{1}{q}\right)\in OQR$ and $T$ be the operator defined in \eqref{eq:T opeartor}. Then for any $\left(\frac{1}{r},\frac{1}{s}\right)\in OQR$ with $\frac{n-1}{s}-\frac{1}{r}=\frac{n-1}{q}-\frac{1}{p}$, we have 
 \begin{align*}
     \|Tg\|_{L^{s}(|x|^{-n})} \lesssim \|g\|_{L^{r,1}(|x|^{-n})}.
 \end{align*}
\end{lemma}

\begin{proof}
    We note that for any $m\geq 100$, due to support of $|g|\chi_{A_{k-m}}$ we can reduce the full fractional maximal operator to the local spherical maximal operator as 
    \begin{align*}
        \chi_{A_k}(x)
 A^{\alpha}_{*}\big(|\cdot|^{-\beta_+-\frac{n}{p}}|g|\chi_{A_{k-m}}\big)(x)&=\chi_{A_k}(x)
 \sup_{2^{k-1}\le t\le 2^{k+2}}A^{\alpha}_{t}\big(|\cdot|^{-\beta_+-\frac{n}{p}}|g|\chi_{A_{k-m}}\big)(x).
    \end{align*}
For any function $f$, we can rescale the above maximal function and write
		\begin{align}
			\sup_{t\in [2^{k-1},2^{k+2}]}|A_t^\alpha f(x)|\le& 2^{(k-1)\alpha}\widetilde{A}_{*}\big((f)_{2^{k-1}}\big)(2^{-k+1}x)+2^{k\alpha}\widetilde{A}_{*}\big((f)_{2^k}\big)(2^{-k}x)\nonumber\\
            &+2^{(k+1)\alpha}\widetilde{A}_{*}\big((f)_{2^{k+1}}\big)(2^{-k-1}x),\label{eq:rescaling}
		\end{align}
        where we use the notation $(f)_{s}=f(sx)$.

    By a change of variables in \eqref{eq:rescaling} and using Lemma \ref{localsmallsuport} along with Remark \ref{rem:dim2}, we obtain
    \begin{align*}
        &\left\||\cdot|^{\beta_++\frac{n}{q}}\chi_{A_k}\sup_{2^{k+j}\le t\le 2^{k+j+1}}A^{\alpha}_{t}\big(|\cdot|^{-\beta_+-\frac{n}{p}}|g|\chi_{A_{k-m}}\big)\right\|_{L^s(|x|^{-n})}\\
        &\lesssim 2^{k\big(\beta_++\frac{n}{q}-\frac{n}{s}\big)}2^{k\big(\alpha+\frac{n}{s}-\frac nr\big)}2^{-m\big(n-1+\frac{n-1}{s}-\frac{n+1}{r}\big)}\||\cdot|^{-\beta_+-\frac{n}{p}}|g|\chi_{A_{k-m}}\|_{L^r}\\
        &\lesssim 2^{k\big(\beta_++\frac{n}{q}+\alpha-\frac nr\big)}2^{-m\big(n-1+\frac{n-1}{s}-\frac{n+1}{r}\big)}2^{-(k-m)\big(\beta_++\frac np-\frac nr\big)}\||g|\chi_{A_{k-m}}\|_{L^r(|x|^{-n})}\\
        &= 2^{-m\big(\frac{n-1}{s}-\frac{1}{r}-\frac{n-1}{q}+\frac{1}{p}\big)}\||g|\chi_{A_{k-m}}\|_{L^r(|x|^{-n})}.
    \end{align*}
    Using the above estimate and the embedding $\ell_r\hookrightarrow\ell_s$, we obtain
    \begin{align}
        \|T_mg\|_{L^s(|x|^{-n})}&= \left(\sum_{k\in\Z}\left\||\cdot|^{\beta_++\frac{n}{q}}\chi_{A_k}\sup_{2^{k+j}\le t\le 2^{k+j+1}}A^{\alpha}_{t}\big(|\cdot|^{-\beta_+-\frac{n}{p}}|g|\chi_{A_{k-m}}\big)\right\|_{L^s(|x|^{-n})}^s\right)^{\frac 1s}\nonumber\\
        &\lesssim \left(\sum_{k\in\Z}2^{-ms\big(\frac{n-1}{s}-\frac{1}{r}-\frac{n-1}{q}+\frac{1}{p}\big)}\||g|\chi_{A_{k-m}}\|_{L^r(|x|^{-n})}^s\right)^{\frac 1s}\nonumber\\
        &\leq 2^{-m\big(\frac{n-1}{s}-\frac{1}{r}-\frac{n-1}{q}+\frac{1}{p}\big)}\left(\sum_{k\in\Z}\||g|\chi_{A_{k-m}}\|_{L^r(|x|^{-n})}^r\right)^{\frac 1r}\nonumber\\
        &= 2^{-m\big(\frac{n-1}{s}-\frac{1}{r}-\frac{n-1}{q}+\frac{1}{p}\big)}\|g\|_{L^r(|x|^{-n})}.\label{eq:Tm}
    \end{align}

    Let us fix $\left(\frac 1r,\frac 1s\right)\in OQR$ satisfying $\frac{n-1}{s}-\frac{1}{r}=\frac{n-1}{q}-\frac{1}{p}$. Note that we do not obtain any decay or growth in the above $L^r\to L^s$ estimate of the operators $T_m$. However, for any such $\left(\frac 1r,\frac 1s\right)$ there exists $\delta>0$ such that the points $\left(\frac{1}{r_1},\frac 1s\right)=\left(\frac 1r-\delta,\frac 1s\right)$ and $\left(\frac{1}{r_2},\frac 1s\right)=\left(\frac 1r+\delta,\frac 1s\right)$ lie inside $OQR$. For these points, from \eqref{eq:Tm}, we get
    \begin{align}
        \|T_mg\|_{L^s(|x|^{-n})}&\lesssim 2^{m\delta}\|g\|_{L^{r_1}(|x|^{-n})},\label{eq:Tm+}\\
        \|T_mg\|_{L^s(|x|^{-n})}&\lesssim 2^{-m\delta}\|g\|_{L^{r_2}(|x|^{-n})}.\label{eq:Tm-}
    \end{align}
    If we take $g=\chi_E$ for a finite measurable set $E$, then from \eqref{eq:Tm+} and \eqref{eq:Tm-}, we get
    \begin{align*}
        \|T\chi_E\|_{L^s(|x|^{-n})}&\lesssim \sum_{m\geq 100}\min\{2^{m\delta}\mu(E)^{\frac 1r-\delta},2^{-m\delta}\mu(E)^{\frac 1r+\delta}\}\\
        &\lesssim \mu(E)^{\frac 1r},
    \end{align*}
    where the measure $\mu$ is given by $d\mu=|x|^{-n}\,dx$. Using a standard argument, one can extend the the above estimate for any $g\in L^{r,1}(\mu)$, i.e.
    \[\|Tg\|_{L^s(|x|^{-n})}\lesssim \|g\|_{L^{r,1}(|x|^{-n})}.\]
    This completes the proof of the lemma.
\end{proof}

	\section{Single Average}\label{single average}
	\begin{theorem}\label{thm:single average}
		Let $n\geq 2$ and $1< p\leq q<\infty$ with $\frac{\alpha}{n}=\frac{1}{p}-\frac{1}{q}$. If $\big(\frac{1}{p},\frac{1}{q}\big)$ lies in the interior of the region $OAB$ or the open line segment $OA$, then
		\[\|A_{t}^{\alpha} f\|_{L^q(|x|^{q\beta})} \leq C \|f\|_{L^p(|x|^{p\beta})},\]
		if and only if $-\frac{n}{q}+\frac{1}{p}\leq \beta \leq n - 1 + \frac{1}{q} - \frac{n}{p}$.
	\end{theorem}
	\begin{proof}
		We will prove the weighted estimate for $t=1$, then a scaling argument implies the estimate for general $t>0$. When $\big(\frac{1}{p},\frac{1}{q}\big)$ lies on the open line segment $OA$, the result was proved in \cite[Proposition 3]{DV}. For $p=\frac{n+1}{n}$ and $q=n+1$, we have $-\frac{n}{q}+\frac{1}{p}=0=n - 1 + \frac{1}{q} - \frac{n}{p}$, and
		\[\|A_{1}^{\alpha} f\|_{L^{n+1}} \leq C \|f\|_{L^\frac{n+1}{n}}.\]
		Interpolating the above estimate with the known weighted estimate when $\big(\frac{1}{p},\frac{1}{q}\big)$ lies on the open line segment $OA$ proves the result for $A_1^\alpha$. Now, we give the scaling argument. We can rewrite $A^\alpha_t$ as
		$$ A_t^\alpha f(x) = t^\alpha A_1 h\Big(\frac{x}{t}\Big), $$
		where $h(y)=f(ty)$. Then, by a change of variable and boundedness for $A_1$, we get
		\begin{align*}
			\int_{\mathbb{R}^n} |A_t^\alpha f(x)|^q |x|^{q\beta} \, dx &= \int_{\mathbb{R}^n} t^{q\alpha} |A_1 h(x)|^q (t |x|)^{q\beta} t^n \, dx \\
			&\lesssim t^{q\alpha + q\beta + n} \left(\int_{\mathbb{R}^n} |h(x)|^p |x|^{p\beta} \, dx\right)^{\frac{q}{p}}\\
			&= t^{q\alpha + q\beta + n} t^{- q\beta - \frac{nq}{p}} \|f\|_{L^p(|x|^{p\beta})}^q\\
			&= \|f\|_{L^p(|x|^{p\beta})}^q,
		\end{align*}
		where we use the fact that $\alpha = \frac{n}{p} - \frac{n}{q}$.
	\end{proof}
	
	\section{Lacunary fractional spherical maximal function}\label{lacunary}
	In this section we give the

	\begin{proof}[Proof of Theorem \ref{thm:lac}]
		We first consider $\beta>0$. In view of Lemma \ref{direct}, it is enough to prove 
		\begin{align*}
			\left\|\sum_{k\in\mathbb Z}\chi_{A_k}A_{lac}^\alpha\Big(\sum_{l< k-100}f_l\Big)\right\|_{L^q(|x|^{q\beta})}\lesssim \|f\|_{L^p(|x|^{p\beta})}.
		\end{align*}
		Fix $k$ and $l< k-100$. Note that for $x\in A_k$, due to the support of $f_l$ in $A_l$, we can see that
		\[A_{lac}^\alpha f_l(x) \sim |A_{2^k}^\alpha f_l(x)|.\]
		By a scaling argument, we can rewrite
		\[A_{2^k}^\alpha f_l(x)=2^{k\alpha}A_1(f_l)_{2^k}(2^{-k}x),\]
		where $(f_l)_{2^k}(y)=f_l(2^ky)$ and $\supp((f_l)_{2^k})\subset A_{l-k}$.
		Since, $l-k<-100$, a change of variable and Lemma~\ref{localsmallsuport} with $\varepsilon=2^{l-k}$ gives
		\begin{align*}
			\|\chi_{A_k}A_{lac}^\alpha f_l\|_{L^q(|x|^{q\beta})} &\lesssim 2^{k(\alpha+\beta)} 2^{\frac{nk}{q}} 2^{(l-k)(n - 1 + \frac{1}{q} - \frac{n}{p})} \|(f_l)_{2^k}\|_{L^p}\\
			&= 2^{k(\alpha+\beta+\frac{n}{q})} 2^{(l-k)(n - 1 + \frac{1}{q} - \frac{n}{p})} 2^{-\frac{nk}{p}}\|f_l\|_{L^p}\\
			&\lesssim 2^{(l-k)(n - 1 + \frac{1}{q} - \frac{n}{p}-\beta)} \|f_l\|_{L^p(|x|^{p\beta})}
		\end{align*}
		where we used the fact that $\alpha=\frac{n}{p}-\frac{n}{q}$.
		
		Therefore, Young's inequality and the embedding $\ell_p\hookrightarrow \ell_q$ yield
		\begin{align*}
			\left\|\sum_{k\in\mathbb Z}\chi_{A_k}A_{lac}^\alpha\Big(\sum_{l< k-100}f_l\Big)\right\|_{L^q(|x|^{q\beta})} &\le \left(\sum_{k\in\Z}\left(\sum_{l\le k-100}\Big\|\chi_{A_k}A_{lac}^{\alpha}f_{l}\Big\|_{L^q(|x|^{q\beta})}\right)^q\right)^{\frac 1q}\\
			&\lesssim \left(\sum_{k\in\Z}\left(\sum_{l\le k-100}2^{-(k-l)(n - 1 + \frac{1}{q} - \frac{n}{p}-\beta)} \|f_{l}\|_{L^p(|x|^{p\beta})}\right)^q\right)^{\frac 1q}\\
			&\leq \sum_{k\geq 100}2^{-k(n - 1 + \frac{1}{q} - \frac{n}{p}-\beta)}\left(\sum_{l\in\Z}\|f_{l}\|_{L^p(|x|^{p\beta})}^q\right)^{\frac 1q}\\
			&\lesssim \left(\sum_{l\in\Z}\|f_{l}\|_{L^p(|x|^{p\beta})}^p\right)^{\frac 1p}\\
			&= \|f\|_{L^p(|x|^{p\beta})}.
		\end{align*}
		This completes the proof for positive exponents whenever $\beta<n-1+\frac{1}{q}-\frac{n}{p}$. 
		
		We next prove the negative endpoint estimate
		\[\|A_{lac}^\alpha f\|_{L^q(|x|^{-q\beta_-})} \lesssim \|f\|_{L^p(|x|^{-p\beta_-})},\]
		where $\beta_-=\frac nq-\frac 1p$ and $\big(\frac{1}{p},\frac{1}{q}\big)\in OAB$. Then, complex interpolation between the above estimate and positive exponent power-weights estimate proves Theorem~\ref{thm:lac}.
		
		Again by Lemma \ref{direct}, it remains to estimate the term
		\begin{align*}
			\left\|\sum_{k\in\Z}\chi_{A_k}A_{lac}^{\alpha} \Big(\sum_{l>k+100}f_{l}\Big)\right\|_{L^q(|x|^{-q\beta_-})}.
		\end{align*}
		Consider a linearisation of $A_{lac}^{\alpha}$ given by
		\[Tf(x):=2^{k(x)\alpha} \sigma_{2^{k(x)}}*f(x),\]
		for a measurable integer-valued function $k(x)$. For $j\in\mathbb Z$, define
		\[E_j:=\{x:k(x)=j\}.\]
		Note that the sets $E_j$ are pairwise disjoint. Then, we can write
		\begin{align*}
			\sum_{k\in\Z}\chi_{A_k}T \Big(\sum_{l>k+100}f_{l}\Big) &= \sum_{j\in\Z}\sum_{k\in\Z}\sum_{l>k+100}\chi_{E_j}\chi_{A_k}T f_{l}\\
			&= \sum_{j\in\Z}\sum_{l \in\Z}\sum_{k<l-100}\chi_{E_j}\chi_{A_k}T f_{l}\\
			&= \sum_{j\in\Z}\sum_{l \in\Z}\chi_{E_j}\chi_{B(0,2^{l-100})}T f_{l}.
		\end{align*}
		For a fixed $j$, if $x\in E_j\cap B(0,2^{l-100})$, we have $k(x)=j$ and the sphere $x-2^{k(x)}S^{n-1}$ has a non-empty intersection with the annulus $A_l$ only if $2^l= 2^j$ with
		\[|x-2^{k(x)}y|\sim 2^j, \quad \forall\, y\in S^{n-1}.\]
		Thus, on the set $E_j$, only $O(1)$ many functions $f_l$ will contribute. Therefore, we obtain
		\begin{align*}
			\sum_{k\in\Z}\chi_{A_k}T \Big(\sum_{l>k+100}f_{l}\Big) &\lesssim \sum_{l \in\Z}\chi_{E_{l}}\chi_{B(0,2^{l-100})}T f_{l}.
		\end{align*}
		
		For each $l$, when $x\in E_l$, by scaling, we obtain
		\[T f_{l}(x) = 2^{l\alpha} ~A_1(f_l)_{2^l}(2^{-l}x),\quad x\in E_l.\]
		Using the negative endpoint power-weight estimate for a single average in Theorem \ref{thm:single average} along with a scaling argument, we obtain
		\[\|\chi_{E_{l}}\chi_{B(0,2^{l-100})}T f_{l}\|_{L^q(|x|^{-q\beta_-})}\lesssim\|f_{l}\|_{L^p(|x|^{-p\beta_-})}.\]
		Therefore, using the embedding $\ell_p\hookrightarrow \ell_q$, we have
		\begin{align*}
			\left\|\sum_{l \in\Z}\chi_{E_{l}}\chi_{B(0,2^{l-100})}T f_{l}\right\|_{L^q(|x|^{-q\beta_-})}&= \left(\sum_{l \in\Z}\|\chi_{E_{l}}\chi_{B(0,2^{l-100})}T f_{l}\|_{L^q(|x|^{-q\beta_-})}^q\right)^{\frac 1q}\\
			&\lesssim \left(\sum_{l \in\Z}\|f_{l}\|_{L^p(|x|^{-p\beta_-})}^q\right)^{\frac 1q}\\
			&\leq \left(\sum_{l \in\Z}\|f_{l}\|_{L^p(|x|^{-p\beta_-})}^p\right)^{\frac 1p}\\
			&= \|f\|_{L^p(|x|^{-p\beta_-})}.
		\end{align*}
		Since the above estimate is independent of the linearization $T$, we obtain
		\[\left\|\sum_{k\in\Z}\chi_{A_k}A_{lac}^{\alpha} \Big(\sum_{l>k+100}f_{l}\Big)\right\|_{L^q(|x|^{-q\beta_-})}\lesssim \|f\|_{L^p(|x|^{-p\beta_-})}.\]
	\end{proof}
	
	\section{Full fractional spherical maximal function}\label{full}
	In this section we give the
	\begin{proof}[Proof of Theorem \ref{thm:full}]
		We first consider $\beta>0$. By Lemma \ref{direct}, it is enough to prove
		\begin{align*}
			\left\|\sum_{k\in\mathbb Z}\chi_{A_k}A_*^\alpha\Big(\sum_{l< k-100}f_l\Big)\right\|_{L^q(|x|^{q\beta})}\lesssim \|f\|_{L^p(|x|^{p\beta})}.
		\end{align*}
		We fix $k$ and $l< k-100$. Note that for $x\in A_k$, due to the support of $f_l$ in $A_l$, we have
		\[A_*^\alpha f_l(x) \leq \sup_{t\in [2^{k-1},2^{k+2}]}|A_t^\alpha f_l(x)|.\]
		Following the rescaling argument as in \eqref{eq:rescaling}, we get
		\begin{align*}
			\sup_{t\in [2^{k-1},2^{k+2}]}|A_t^\alpha f_l(x)|\le& 2^{(k-1)\alpha}\widetilde{A}_{*}\big((f_l)_{2^{k-1}}\big)(2^{-k+1}x)+2^{k\alpha}\widetilde{A}_{*}\big((f_l)_{2^k}\big)(2^{-k}x)\\
			&+2^{(k+1)\alpha}\widetilde{A}_{*}\big((f_l)_{2^{k+1}}\big)(2^{-k-1}x).
		\end{align*}
		Thus, by a change of variable argument and Lemma~\ref{localsmallsuport}, where by relabeling we replace $\gamma-\eta$ with $\gamma$, we obtain
		\begin{align*}
			\|\chi_{A_k}A_*^\alpha f_l\|_{L^q(|x|^{q\beta})} &\lesssim 2^{k(\alpha+\beta)} 2^{\frac{nk}{q}} 2^{(l-k)\gamma} \|(f_l)_{2^k}\|_{L^p}\\
			&= 2^{k(\alpha+\beta+\frac{n}{q})} 2^{(l-k)\gamma} 2^{-\frac{nk}{p}}\|f_l\|_{L^p}\\
			&\lesssim 2^{(l-k)(\gamma-\beta)} \|f_l\|_{L^p(|x|^{p\beta})}
		\end{align*}
		where we used the fact that $\alpha=\frac{n}{p}-\frac{n}{q}$.
		
		When $\beta<\gamma$, the choice of $\eta>0$ in Lemma~\ref{localsmallsuport} can be made such that $\beta$ is strictly less than the relabeled $\gamma$.  Therefore, Young's inequality and the embedding $\ell_p\hookrightarrow \ell_q$ give
		\begin{align*}
			\left\|\sum_{k\in\mathbb Z}\chi_{A_k}A_*^\alpha\Big(\sum_{l< k-100}f_l\Big)\right\|_{L^q(|x|^{q\beta})} &\le \left(\sum_{k\in\Z}\left(\sum_{l< k-100}\Big\|\chi_{A_k}A_{*}^{\alpha}f_{l}\Big\|_{L^q(|x|^{q\beta})}\right)^q\right)^{\frac 1q}\\
			&\lesssim \left(\sum_{k\in\Z}\left(\sum_{l< k-100}2^{-(k-l)(\gamma-\beta)} \|f_{l}\|_{L^p(|x|^{p\beta})}\right)^q\right)^{\frac 1q}\\
			&\leq \sum_{k> 100}2^{-k(\gamma-\beta)}\left(\sum_{l\in\Z}\|f_{l}\|_{L^p(|x|^{p\beta})}^q\right)^{\frac 1q}\\
			&\lesssim \left(\sum_{l\in\Z}\|f_{l}\|_{L^p(|x|^{p\beta})}^p\right)^{\frac 1p}\\
			&= \|f\|_{L^p(|x|^{p\beta})}.
		\end{align*}
        This proves the optimal positive exponent power weight estimates for the full fractional spherical maximal function, except for the upper endpoint in the region $OQR$. Now, we give the proof for the upper endpoint in the region $OQR$. 
        
        Let us recall the upper endpoint $\beta_+=n-1+\frac{n-1}{q}-\frac{n+1}{p}$ in the region $OQR$. Then, from Lemma \ref{direct}, it remains to prove
        \begin{align*}
			\left\|\sum_{k\in\mathbb Z}\chi_{A_k}A_*^\alpha\Big(\sum_{l< k-100}f_l\Big)\right\|_{L^q(|x|^{q\beta_+})}\lesssim \|f\|_{L^p(|x|^{p\beta_+})}.
		\end{align*}
        By relabeling, we can write
        \begin{align}
            &\left\|\sum_{k\in\mathbb Z}\chi_{A_k}A_*^\alpha\Big(\sum_{l< k-100}f_l\Big)\right\|_{L^q(|x|^{q\beta_+})}\nonumber\\
            &\leq \left\|\sum_{k\in\mathbb Z}\sum_{m>100}\chi_{A_k}A_*^\alpha\Big(|f_{k-m}|\Big)\right\|_{L^q(|x|^{q\beta_+})}\nonumber\\
            &= \left\|\sum_{m>100}\sum_{k\in\mathbb Z}|\cdot|^{\beta_++\frac nq}\chi_{A_k}A_*^\alpha\Big(|\cdot|^{-\beta_+-\frac np}|\cdot|^{\beta_++\frac np}|f_{k-m}|\Big)\right\|_{L^q(|x|^{-n})}\nonumber\\
            &= \|T(|\cdot|^{\beta_++\frac np}|f|)\|_{L^q(|x|^{-n})}.\label{eq:T reduction}
        \end{align}
        Note that, whenever $\left(\frac 1p,\frac 1q\right)\in OQR$, we can find $\delta>0$ sufficiently small and $\left(\frac{1}{r_0},\frac{1}{s_0}\right), \left(\frac{1}{r_1},\frac{1}{s_1}\right)\in OQR$ such that
\begin{align*}
\frac1{r_0}&=\frac1p+\delta,&\frac1{s_0}&=\frac1q+\frac{\delta}{n-1},\\
 \frac1{r_1}&=\frac1p-\delta,& \frac1{s_1}&=\frac1q-\frac{\delta}{n-1},
\end{align*}
 and $r_0,s_0,r_1,s_1$ satisfy the relation
\[\frac{n-1}{s_0}-\frac{1}{r_0}=\frac{n-1}{q}-\frac{1}{p}=\frac{n-1}{s_1}-\frac{1}{r_1}.\]
Therefore, applying Lemma \ref{lem:upper endpoint}, we obtain
\begin{align}
    \|Tg\|_{L^{s_0}(|x|^{-n})}&\lesssim \|g\|_{L^{r_0,1}(|x|^{-n})}\label{eq:r0s0}\\
    \|Tg\|_{L^{s_1}(|x|^{-n})}&\lesssim \|g\|_{L^{r_1,1}(|x|^{-n})}.\label{eq:r1s1}
\end{align}
A real interpolation between \eqref{eq:r0s0} and \eqref{eq:r1s1} gives
\[\|Tg\|_{L^{q}(|x|^{-n})}\lesssim \|g\|_{L^{p}(|x|^{-n})}.\]
Substituting the above estimate in \eqref{eq:T reduction} implies
\begin{align*}
    \left\|\sum_{k\in\mathbb Z}\chi_{A_k}A_*^\alpha\Big(\sum_{l< k-100}f_l\Big)\right\|_{L^q(|x|^{q\beta_+})}&\lesssim \||\cdot|^{\beta_++\frac np}f\|_{L^p(|x|^{-n})}\\
    &= \|f\|_{L^p(|x|^{p\beta_+})}.
\end{align*}
This completes the proof for positive power exponents.

		Next, we prove the negative exponent power-weight estimate
		\[\|A_{*}^\alpha f\|_{L^q(|x|^{-q\beta_-})} \lesssim \|f\|_{L^p(|x|^{-p\beta_-})}.\]
		Then, complex interpolation between the above estimate and positive exponent power-weighted estimate proves Theorem~\ref{thm:full}. In view of Lemma \ref{direct}, it is enough to estimate the term
		\begin{align*}
			\left\|\sum_{k\in\Z}\chi_{A_k}A_{*}^{\alpha} \Big(\sum_{l>k+100}f_{l}\Big)\right\|_{L^q(|x|^{-q\beta_-})}.
		\end{align*}
		We linearise the maximal operator $A_{*}^{\alpha}$ by
		\[Tf(x):=t(x)^\alpha \sigma_{t(x)}*f(x),\]
		where $t(x)$ is a measurable function, and define
		\[E_j:=\{x:2^j\leq t(x)<2^{j+1}\}.\]
		Note that the sets $E_j$ are pairwise disjoint. As in the lacunary case, we obtain
		\begin{align*}
			\sum_{k\in\Z}\chi_{A_k}T \Big(\sum_{l>k+100}f_{l}\Big) 
			&= \sum_{j\in\Z}\sum_{l \in\Z}\chi_{E_j}\chi_{B(0,2^{l-100})}T f_{l}.
		\end{align*}
		If $x\in E_j\cap B(0,2^{l-100})$ and the sphere $x-t(x)S^{n-1}$ meets the annulus $A_l$, then $2^j\sim t(x)\sim 2^l$. Consequently only three values of $l$ contribute for each $j$. Therefore, we obtain
		\begin{align*}
			\sum_{k\in\Z}\chi_{A_k}T \Big(\sum_{l>k+100}f_{l}\Big) &\lesssim \sum_{l \in\Z}\chi_{E_{l}}\chi_{B(0,2^{l-100})}T f_{l}.
		\end{align*}
		
		When $x\in E_l$, by scaling, we can write
		\[T f_{l}(x) = 2^{l\alpha}s(x)^\alpha (f_l)_{2^l}*\sigma_{s(x)}(2^{-l}x),\]
		where $t(x)=2^{l} s(x)$ with $s(x)\in[1,2]$. Then
		\[|T f_{l}(x)|\lesssim 2^{l\alpha}\widetilde{A}_{*}(f_l)_{2^l}(2^{-l}x).\]
		We claim the following estimate for $\widetilde{A}_{*}(f_l)_{2^l}$ when $\left(\frac{1}{p},\frac{1}{q}\right)\in OPQR\cup OP$.
		\begin{equation}\label{claimlocal}
			\|\chi_{B(0,2^{-100})}\widetilde{A}_{*} (f_l)_{2^l}\|_{L^q(|x|^{-q\beta_-})}\lesssim \|(f_l)_{2^l}\|_{L^p(|x|^{-p\beta_-})}\sim \|(f_l)_{2^l}\|_{L^p}.
		\end{equation}
		Let us assume \eqref{claimlocal} for the moment. Then, by rescaling, we obtain
		\begin{align*}
			\|\chi_{E_{l}}\chi_{B(0,2^{l-100})}T f_{l}\|_{L^q(|x|^{-q\beta_-})} &\lesssim 2^{l\alpha}2^{l (n/q-\beta_-)}\|(f_l)_{2^l}\|_{L^p}\\
			&\lesssim 2^{l (\alpha+n/q-\beta_-)}2^{-l n/p}\|f_{l}\|_{L^p}\\
			&\lesssim \|f_{l}\|_{L^p(|x|^{-p\beta_-})},
		\end{align*}
		where we used the condition $\alpha=n/p-n/q$. Since the sets $E_l$ are pairwise disjoint, using the embedding $\ell_p\hookrightarrow \ell_q$, we have
		\begin{align*}
			\left\|\sum_{l \in\Z}\chi_{E_{l}}\chi_{B(0,2^{l-100})}T f_{l}\right\|_{L^q(|x|^{-q\beta_-})}&= \left(\sum_{l \in\Z}\|\chi_{E_{l}}\chi_{B(0,2^{l-100})}T f_{l}\|_{L^q(|x|^{-q\beta_-})}^q\right)^{\frac 1q}\\
			&\lesssim \left(\sum_{l \in\Z}\|f_{l}\|_{L^p(|x|^{-p\beta_-})}^q\right)^{\frac 1q}\\
			&\leq \left(\sum_{l \in\Z}\|f_{l}\|_{L^p(|x|^{-p\beta_-})}^p\right)^{\frac 1p}\\
			&= \|f\|_{L^p(|x|^{-p\beta_-})}.
		\end{align*}
		Since the above estimate is independent of the linearization $T$, we obtain
		\[\left\|\sum_{k\in\Z}\chi_{A_k}A_{*}^{\alpha} \Big(\sum_{l>k+100}f_{l}\Big)\right\|_{L^q(|x|^{-q\beta_-})}\lesssim \|f\|_{L^p(|x|^{-p\beta_-})}.\]
		This completes the proof barring the estimate \eqref{claimlocal}.
		
		\subsection*{Proof of the claimed inequality \eqref{claimlocal}.}
		We will prove the claim in two parts. First, we will prove for $q\geq2$ with Littlewood-Paley decomposition and then we will extend to the remaining range by H\"older's inequality along with an interpolation argument for weighted estimates. Suppose that $\supp h\subset A_0$. To prove the claim in \eqref{claimlocal}, it is enough to show that
		\begin{align}\label{weighted_endpoint_lebesgue}
		\|\widetilde{A}_{*} h\|_{L^q(B(0,2^{-100}),|x|^{-q\beta_-})}\lesssim \|h\|_{L^p}.
		\end{align}
		
		\subsection*{When $q\geq2$.} Using the decomposition of $\widetilde{A}_{*}$ in \eqref{eq:deco}, we have
		\[\|\widetilde{A}_{*} h\|^q_{L^q(B(0,2^{-100}),|x|^{-q\beta_-})}\lesssim  \left\|M_0h\right\|^q_{L^q(B(0,2^{-100}),|x|^{-q\beta_-})} + \sum_{k<-100} \left\|\chi_{A_k}\sum_{j\geq1}M_{j}h\right\|^q_{L^q(|x|^{-q\beta_-})}.\]
        Since, the multiplier associated to the maximal operator $M_0$ is a compactly supported smooth function, therefore, we have
		\[M_0h\lesssim \mathcal{M}^\alpha h.\]
		Thus, power-weight estimates of the fractional maximal function $\mathcal M^\alpha$ give the desired estimate for low frequency term. 
		
		For high-frequency terms, Lemma \ref{demposedlocalweighted} gives
		\begin{align*}
			\sum_{k<-100} \left(\sum_{j\geq1} \left\|\chi_{A_k}M_{j}h\right\|_{L^q(|x|^{-q\beta_-})}\right)^q&\lesssim \sum_{k<-100} \left(\sum_{j\geq1}2^{-\epsilon|j+k|}\|h\ast\psi_{2^j}\|_{L^p}\right)^q\\
			&\leq \left(\sum_{k\in\Z}2^{-\epsilon |k|} \left(\sum_{j\geq1}\|h\ast\psi_{2^j}\|_{L^p}^q\right)^{1/q}\right)^q,
		\end{align*}
		where we applied the Young's inequality in the last step and relabeled $k$ for $-k$.  
		
		Since $p\leq q$ and $q\geq2$, Minkowski's integral inequality and the embedding $\ell_2\hookrightarrow \ell_q$ yield
		\begin{align*}
			\sum_{j\geq1}\|h\ast\psi_{2^j}\|_{L^p}^q&\leq \left\|\left(\sum_{j\geq1}|h\ast\psi_{2^j}|^q\right)^{\frac{1}{q}}\right\|_{L^p}^q\\
			&\leq \left\|\left(\sum_{j\geq1}|h\ast\psi_{2^j}|^2\right)^{\frac{1}{2}}\right\|_{L^p}^q\\
			&\lesssim \|h\|_{L^p}^q,
		\end{align*}
		where we used Littlewood-Paley theory in the last inequality. This proves the desired claim in \eqref{claimlocal} for $\left(\frac{1}{p},\frac{1}{q}\right)\in OPQR$ when $q\geq2$.  
		
		As it will be useful to treat the remaining case, let us next record a Lorentz space refinement of the estimate \eqref{weighted_endpoint_lebesgue}. For any $\left(\frac{1}{p},\frac{1}{q}\right)\in OPQR$ with $q> 2$, we can find points $P_0=\left(\frac{1}{p_0},\frac{1}{q_0}\right), P_1=\left(\frac{1}{p_1},\frac{1}{q_1}\right)\in OPQR$ with $q_0,q_1>2$ such that
\begin{align*}
\frac{1}{p} = \frac{1/2}{p_0} + \frac{1/2}{p_1},\quad \frac{1}{q} = \frac{1/2}{q_0} + \frac{1/2}{q_1}
\end{align*}
and
\begin{align*}
\frac{1}{p_0}-\frac{n}{q_0}=\frac{1}{p_1}-\frac{n}{q_1}.
\end{align*}
		 Then by weighted real interpolation between $P_0$ and $P_1$ using Lemma \ref{weightedinterpolation} with $w_0=w_1=|x|^{-\beta_-}$, one has the improved estimate
		\begin{equation}\label{betterlocal}
			\||x|^{-\beta_-}\chi_{B(0,2^{-100})}\widetilde{A}_{*} h\|_{L^{q,p}} = \|\chi_{B(0,2^{-100})}\widetilde{A}_{*} h\|_{L^{q,p}(|x|^{-\beta_-},\, dx)}\lesssim \|h\|_{L^p}.
		\end{equation}

		\subsection*{When $q<2$.}

Let $\frac{n}{n-1}<p\leq q<2$, we choose $r>2$ such that $\left(\frac{1}{p},\frac{1}{r}\right)\in OPQR$. We write
		\[\beta_-=\frac{n}{q}-\frac{n}{r}+\frac{n}{r}-\frac{1}{p}.\]
 H\"older's inequality on the Lorentz scale, the fact that $|x|^{-(n/q-n/r)}\in L^{\frac{qr}{r-q},\infty}$, and the embedding $L^{r,p}\hookrightarrow L^{r,q}$ give
		\begin{align*}
			\|\chi_{B(0,2^{-100})}\widetilde{A}_{*} h\|_{L^q(|x|^{-q\beta_-})}&= \||\cdot|^{-(n/q-n/r)-(n/r-1/p)}\chi_{B(0,2^{-100})}\widetilde{A}_{*}h\|_{L^q}\\
			&\lesssim \||\cdot|^{-(n/q-n/r)}\|_{L^{\frac{qr}{r-q},\infty}}\||\cdot|^{-(n/r-1/p)}\chi_{B(0,2^{-100})}\widetilde{A}_{*}h\|_{L^{r,q}}\\
			&\lesssim \||\cdot|^{-(n/r-1/p)} \chi_{B(0,2^{-100})} \widetilde{A}_{*}h\|_{L^{r,p}}\\
			&\lesssim \|h\|_{L^p}.
		\end{align*}
		This completes the proof of \eqref{claimlocal} for the remaining range $q<2$.
	\end{proof}

	\section{Necessary Conditions for Weighted Boundedness}\label{examples}
	
		\subsection{Single average}
		
		We show that the conditions in Theorem \ref{thm:single average} are necessary.  As noted in the proof, the estimate for $A_t^\alpha$ is equivalent to the one for $A_1^\alpha = A_1$, with the same constant.  
		
		\subsection*{Example 1.}
	Let $0<\delta<1/2$ and set $f =\chi_{A_\delta},$ where $A_\delta = \{y \in \mathbb{R}^n : 1 -\delta \le |y| \le 1+\delta\}$. The weight $|x|^{p\beta}$ is comparable to 1 in $A_\delta$ and $|A_\delta|\sim \delta$. Thus,
	\[\|f\|_{L^p(|x|^{p\beta})} \sim \delta^{1/p}.\]
	For $|x| < \delta/2$, the sphere centered at $x$ of radius $1$ lies completely within the shell $A_\delta$. Thus, $A_{1}^{} f(x) \sim 1$ for all $|x| < \delta/2$. 
	
	Therefore,
	\[\|A_{1}^{} f\|_{L^q(|x|^{q\beta})} \gtrsim \left( \int_{|x| < \delta/2} |x|^{q\beta} dx \right)^{1/q} \sim \delta^{\beta + \frac{n}{q}},\]
	whenever $\beta>-\frac{n}{q}$. If the estimate in Theorem \ref{thm:single average} holds, then $\delta^{\beta + \frac{n}{q}} \lesssim \delta^{1/p}$. Now, taking the limit $\delta \to 0$ gives the necessary condition
	\begin{equation}
		\frac{1}{p} - \frac{n}{q}\le \beta.
	\end{equation}

\subsection*{Example 2.}
Let $f =\chi_{B(0,1)}$ and $E_k = \left\{ x \in \mathbb{R}^n : \left| |x| - 2^k \right| < \frac{1}{2} \right\}.$ Then $\|f\|_{L^p(|x|^{p\beta})} \sim 1$ if $\beta>-\frac{n}{p}$ and $|E_k|\sim 2^{k(n-1)}$.

For $x\in E_k$, the sphere centered at $x$ of radius $2^k$ has a surface area of order $2^{k(n-1)}$ and intersects the unit ball centered at the origin with a measure of some constant. Then, for all $x \in E_k$, we have $A_{2^k}^{\alpha} f(x) \gtrsim 2^{k(1-n+\alpha)}$ and 
	\[\|A_{2^k}^{\alpha} f\|_{L^q(|x|^{q\beta})}^q \gtrsim \int_{E_k} \left(2^{k(1-n+\alpha)}\right)^q |x|^{q\beta} dx \sim  2^{kq(1-n+\alpha+\beta)} \cdot 2^{k(n-1)}.\]
Therefore the boundedness of the single average implies
\[
2^{kq(1-n+\alpha+\beta)} \cdot 2^{k(n-1)} \lesssim 1.
\]
The relation $\alpha = \frac{n}{p} - \frac{n}{q}$ and sending $k \to \infty$ yields the condition
\begin{equation}
\beta \leq n - 1 + \frac{1}{q} - \frac{n}{p}.
\end{equation}

Combining Examples 1 and 2, the estimate in Theorem \ref{thm:single average} holds for some (equivalently, every) $t>0$ only if
\[
\frac1p-\frac nq\le\beta\le n-1+\frac1q-\frac np.
\]

	\subsection{Lacunary fractional spherical maximal function}
Examples 1 and 2 show the same conditions are necessary for the lacunary fractional spherical maximal operator.  We give an example which shows strict inequality for the upper endpoint of the range of $\beta$.
	\subsection*{Example 3.}
	Let $f, E_k$ be as in Example 2.  Then, for all $x \in E_k$, we have $A_{lac}^{\alpha} f(x) \gtrsim 2^{k(1-n+\alpha)}$ and 
	\[\|A_{lac}^{\alpha} f\|_{L^q(|x|^{q\beta})}^q \gtrsim \sum_{k=1}^\infty \int_{E_k} \left(2^{k(1-n+\alpha)}\right)^q |x|^{q\beta} dx \sim \sum_{k=1}^\infty 2^{kq(1-n+\alpha+\beta)} \cdot 2^{k(n-1)}.\]
	This geometric series converges if and only if the exponent is strictly negative, i.e.,
	\[q(1-n+\alpha) + q\beta + n - 1 < 0.\]
	Substituting $\alpha = \frac{n}{p} - \frac{n}{q}$, we obtain the condition
	\begin{equation}
		\beta < n - 1 + \frac{1}{q} - \frac{n}{p}.
	\end{equation}
	
	Combining the above example with Example 1 for the single average, the inequality \eqref{eq:main_lac} holds true for the lacunary fractional spherical maximal function only if
	\[\frac{1}{p} - \frac{n}{q} \le \beta < n - 1 + \frac{1}{q} - \frac{n}{p}.\]

	\subsection{Full fractional spherical maximal function}
Examples 1 and 3 show the same conditions are necessary for the full fractional spherical maximal operator.  We here give two more examples which further restrict the admissible range of $\beta$.
		
	\subsection*{Example 4.}
	Let $f =\chi_{B(0,1)}$. Then $\|f\|_{L^p(|x|^{p\beta})} \sim 1$ if $\beta>-\frac{n}{p}$.
	
	For $|x| \gg 1$ and $t = |x|$, the sphere centered at $x$ of radius $t$ has a surface area of order $|x|^{n-1}$ and intersects the unit ball centered at the origin with a measure of some constant. Thus, $A_{*}^{\alpha} f(x) \gtrsim |x|^{1-n+\alpha}$ and
	\[\|A_{*}^{\alpha} f\|_{L^q(|x|^{q\beta})}^q\gtrsim \int_{|x|>2} \left( |x|^{1-n+\alpha} \right)^q |x|^{q\beta} dx = \int_2^\infty r^{(1-n+\alpha)q + q\beta + n - 1} dr.\]
	The above integral is finite only if $(1-n+\alpha)q + q\beta + n < 0$, which gives the condition
	\begin{equation}
		\beta < n - 1 - \frac{n}{p}.
	\end{equation}
	
	\subsection*{Example 5.}
	Let $f =\chi_{R_\delta}$, where $R_\delta = [\delta,C\delta]^{n-1}\times[-\delta^2,\delta^2]$. Then $|R_\delta|=C'\delta^{n+1}$, and $\|f\|_{L^p(|x|^{p\beta})} \sim \delta^{\beta + \frac{n+1}{p}}$.
	
	For $x\in E_\delta = [\delta,2\delta]^{n-1}\times[1,2]$, we choose $t= x_n$. Then, the sphere centered at $x$ of radius $t$ captures an intersection area of $\delta^{n-1}$ with $R_\delta$. When $x\in E_\delta$ the weight $|x|^{q\beta} \sim 1$. Thus, we get
	\[\|A_{*}^{\alpha} f\|_{L^q(|x|^{q\beta})} \gtrsim \left( \int_{E_\delta} (\delta^{n-1})^q dx \right)^{1/q} \sim \delta^{(n-1)\left(1 + \frac{1}{q}\right)}.\]
	For the inequality \eqref{eq:main} to hold, we need $\delta^{(n-1)\left(1 + \frac{1}{q}\right)} \lesssim \delta^{\beta + \frac{n+1}{p}}$. Now taking the limit $\delta \to 0$, we get the condition
	\begin{equation}
		\beta \le n-1 + \frac{n-1}{q} - \frac{n+1}{p}.
	\end{equation}
	
	Combining all the examples, we get that \eqref{eq:main} holds true only if
	\[\frac{1}{p} - \frac{n}{q} \le \beta <  n - 1 - \frac{n}{p}\]
	whenever $\left(\frac{1}{p},\frac{1}{q}\right)\in OPQ\cup OQ$, and
	\[\frac{1}{p} - \frac{n}{q} \le \beta \le n-1 + \frac{n-1}{q} - \frac{n+1}{p},\]
	whenever $\left(\frac{1}{p},\frac{1}{q}\right)\in OQR$.

	\section*{Acknowledgements}
	R.~Basak is supported by the National Science and Technology Council of Taiwan under research grant number 113-2811-M-003-014.  
	S. Choudhary is supported by the National Center for Theoretical Sciences through National Science and Technology Council of Taiwan under research grant number 115-2124-M-002-009-.
	C.-Y. Shen is supported by the National Science and Technology Council of Taiwan through grant 111-2115-M-002-010-MY5.
	D. Spector is supported by the National Science and Technology Council of Taiwan under research grant number 113-2115-M-003-017-MY3.
	
	\section*{AI Disclosure}
	
The authors utilised Claude Opus to generate TikZ Figure 3 based on Figures 1,2 from their preceding manuscript.  The use of H\"older's inequality on the Lorentz scale to extend the inequality \eqref{betterlocal} to the range $q< 2$ and the idea of how to obtain the validity of Theorem \ref{thm:full} at the endpoint $\beta = (n-1) + (n-1)/q-(n+1)/p$ in the triangle $OQR$ ($OPR$ when $n=2$) were obtained through a dialogue with ChatGPT.

\end{document}